\documentclass[14pt]{amsart}
\address{\newline{\normalsize Laboratory of AGHA, Moscow Institute of Physics and Technology, 9 Institutskiy per., Dolgoprudny,
Moscow Region, 141701, Russia}
\newline{\it E-mail address}: karzhemanov.iv@mipt.ru}
\usepackage{amscd,amsthm,amsmath,amssymb}
\usepackage[dvips]{graphicx}
\usepackage[matrix,arrow]{xy}

\makeatletter\@addtoreset{equation}{section}\makeatother

\newtheorem{theorem}[equation]{Theorem}
\newtheorem{prop}[equation]{Proposition}
\newtheorem{lemma}[equation]{Lemma}

\newtheorem{theorem-definition}[equation]{Theorem-definition}

\theoremstyle{definition}

\theoremstyle{remark}
\newtheorem{remark}[equation]{Remark}

\makeatletter\@addtoreset{subsection}{equation}\makeatother

\newcommand{\p}{\mathbb{P}}
\newcommand{\cel}{\mathbb{Z}}

\newcommand{\com}{\mathbb{C}}
\newcommand{\af}{\mathbb{A}}

\newcommand{\fie}{{\bf k}}

\begin{document}

\title{On the cut\,-\,and\,-\,paste property of algebraic varieties}

\author{Ilya Karzhemanov}

\begin{abstract}
We show that the above\,-\,named property (after M.\,Larsen and
V.\,Lunts) does not hold in general.
\end{abstract}

\sloppy

\maketitle

\bigskip

\section{Introduction}
\label{section:in}

\refstepcounter{equation}
\subsection{}
\label{subsection:in-1}

Fix some ground field $\fie\subseteq\com$. Consider the ring
$K_0(\text{Var}_{\fie})$ (or simply $K_0(\text{Var})$) of
algebraic $\fie$\,-\,varieties, generated over $\cel$ by
isomorphism classes $[X]$ of all quasi\,-\,projective such $X$,
with multiplication given by $[X]\cdot[Y] := [X\times_{\fie} Y]$
and addition $[X] + [Y]$ being just the formal sum subject to the
relation $[X] = [X\setminus Y] + [Y]$ whenever $Y \subseteq X$ is
a closed subvariety. In particular, the class of a point can be
identified with the unity $1\in K_0(\text{Var}_{\fie})$, so that
$[\af^1\setminus\{0\}] = \mathbb{L} - 1$ for $[\af^1] :=
\mathbb{L}$ and the affine line $\af^1$. Note that
$K_0(\text{Var})$ can also be generated by isomorphism classes
$[S]$ of $\fie$\,-\,\emph{schemes of finite type} $S$, with the
automatic relation $[S] = [S_{\text{red}}]$, where
$S_{\text{red}}$ is the corresponding reduced scheme.

The ring $K_0(\text{Var})$ (or, more precisely, \emph{motivic
measures} with values in $K_0(\text{Var})$ with $\mathbb{L}$
inverted) was used in \cite{max} for example to show (via the
celebrated \emph{motivic integration}) that birationally
isomorphic Calabi\,--\,Yau manifolds have equal Hodge numbers (see
\cite{craw}, \cite{den-loe}, \cite{looij}, \cite{must},
\cite{la-lu}, \cite{la-lu-1} and \cite{ga-shin} for other results,
applications and references). However, the structure of
$K_0(\text{Var})$ is still poorly understood, although it is known
that $K_0(\text{Var})$ is not a domain (see \cite{poo}) and one
has an explicit description of ``\,conical\,'' subrings in
$K_0(\text{Var})$ (see \cite{kol}). See also \cite{la-lu} and
\cite{bondal-la-lu} for a description of quotients of
$K_0(\text{Var})$ by some ideals related to $\mathbb{L}$.

Further, given two $\fie$\,-\,varieties $X$ and $Y$ let us say
(following \cite{la-lu}) that they satisfy the
\emph{cut\,-\,and\,-\,paste property} if there is a decomposition
$X = \displaystyle\coprod_{i=1}^k X_i$ (resp. $Y =
\displaystyle\coprod_{i=1}^k Y_i$), some $k \ge 1$, into disjoint
locally closed subvarieties $X_i \subseteq X$ (resp. $Y_i
\subseteq Y$), with $X_i \simeq Y_i$ for all $1 \le i \le k$. This
implies in particular that $[X] = [Y]$ in $K_0(\text{Var})$. Our
main result is that \emph{the converse does not hold}:

\begin{theorem}
\label{theorem:main} There exist two \emph{smooth projective
$3$\,-\,folds} $X$ and $Y$, satisfying $[X] = [Y]$, but violating
the cut\,-\,and\,-\,paste property.\footnote{~This also answers at
negative Question (b) in \cite[$\bf 3.G'''$]{gro} (cf.
\cite{cha-et-al}).}
\end{theorem}

Recently, after our work appeared online, the paper \cite{bor}
(see also \cite{zah}) has supplied an example showing that
$\mathbb{L}$ is a $0$\,-\,divisor in $K_0(\text{Var})$. In
addition, a pair of \emph{smooth non\,-\,proper
$10$\,-\,dimensional} $X$ and $Y$, violating
cut\,-\,and\,-\,paste, was also constructed in \emph{loc. cit}.
Thus our Theorem~\ref{theorem:main} provides another (tiny)
contribution to this interesting subject (we hope that our
approach is in a way simpler and more direct).

In order to prove Theorem~\ref{theorem:main} it suffices to
exhibit such $X$ and $Y$, not birational to each other, yet
satisfying $[X] = [Y]$. Specifically, the idea is to take
\emph{birationally rigid} $X$ and $Y$ (see e.\,g. \cite{pux}), and
then show that $[X] = [Y]$. For the latter, we want both $X$ and
$Y$ to be ``\,elementary\,'', fibered over $\p^1$, say, in
surfaces all having the same class in $K_0(\text{Var})$ (this is
also in line with \cite{pux}). The equality $[X] = [Y]$ for
certain $X$ and $Y$ can then be established by a direct argument
(see {\ref{subsection:f-1}} and {\ref{subsection:f-2-djdjdj}}
below).

\bigskip

\section{Proof of Theorem~\ref{theorem:main}}
\label{section:f}

\refstepcounter{equation}
\subsection{}
\label{subsection:f-1}

Fix two cubic forms $G := G(x_0,\ldots,x_3)$ and $F :=
F(x_0,\ldots,x_3)$ over $\fie$. We will assume that $G$ is
\emph{generic} and the equation $F = 0$ defines a cubic surface
(in $\p^3$) with an \emph{ordinary double point as the only
singularity}. We also choose $F$ generic among such forms.

Consider the locus $X \subset \p^3\times\p^1$, given by the
equation
\begin{eqnarray}
\label{cubic-sur-x} \alpha(t_0,t_1)G + \beta(t_0,t_1)F = 0
\end{eqnarray}
of bidegree $(3,m)$ for some $m \ge 3$, where $x_i$ (resp. $t_i$)
are projective coordinates on $\p^3$ (resp. $\p^1$). We also
assume both forms $\alpha,\beta$ to be generic.

\begin{lemma}
\label{theorem:x-is-smooth} $X$ is a \emph{smooth} $3$\,-\,fold.
\end{lemma}

\begin{proof}
Indeed, by Bertini's theorem applied to the linear system of
divisors \eqref{cubic-sur-x} in $\p^3\times\p^1$, all possible
singularities of $X$ can only belong to the surface $\frak{B} :=
(G = F = 0) \subset\p^3\times\p^1$. This surface is \emph{smooth}
because it is isomorphic to $\p^1\times$ [the smooth curve $(G = F
= 0) \subset\p^3]$. In particular, locally near every point of
$\frak{B}$ we may assume both $G,F$ to be some analytic
coordinates on $\p^3\times\p^1$. The claim now easily follows by
taking partial derivatives of \eqref{cubic-sur-x} with respect to
$t_i$ and $G,F$.
\end{proof}

Further, $X$ carries a natural fibration $X \longrightarrow \p^1$
in cubic surfaces, induced by the projection of $\p^3\times\p^1$
onto the second factor. We will refer to $X$ as a \emph{pencil}
(of cubic surfaces).

\refstepcounter{equation}
\subsection{}
\label{subsection:f-1-aks}

Consider the open subset
$$
X_0 := X \setminus (t_0GF = 0)\subset X.
$$
and write $\alpha(1,t) = \displaystyle\sum_{i = 0}^m \alpha_i t^i$
(resp. $\beta(1,t) = \displaystyle\sum_{i = 0}^m \beta_i t^i$) for
some $\alpha_i,\beta_i \in \fie$ and the rational function $t :=
t_1/t_0$ on $X$.

Further, let $\af^{m + 1} =
\text{Spec}\,\fie[y_1,\ldots,y_m,\lambda]$ and define the locus
$\mathcal{X}\subset\p^3\times \af^{m + 1}$ by the equation
\begin{eqnarray}
\label{cubic-sur-x-djd} L_{\alpha}G + (L_{\beta}+\lambda)F = 0,
\end{eqnarray}
where $L_{\alpha} := \alpha_0 + \displaystyle\sum_{i = 1}^m
\alpha_i y_i$ (resp. $L_{\beta} := \beta_0 + \displaystyle\sum_{i
= 1}^m \beta_i y_i$). Let us also consider the open subset
$$
\mathcal{X}_0 := \mathcal{X}\setminus (GF = 0)\subset\mathcal{X}.
$$

Note that $X_0\subset\mathcal{X}_0$ as a closed subset given by
the equations $y_1 = t, y_2 = t^2, \ldots, y_m = t^m, \lambda =
0$. There is a finer relation between $X_0$ and $\mathcal{X}_0$.
Namely, let us define a morphism $\varphi: X_0 \times \af^{m - 1}
\longrightarrow \mathcal{X}_0$ as follows:
\begin{equation}
\label{mor-1} [x_0:\ldots:x_3] \times (t,y_2,\ldots,y_m) \mapsto
[x_0:\ldots:x_3]\times (y'_1,y'_2,\ldots,y'_m,\lambda'),
\end{equation}
with
\begin{equation}
\label{mor-2} y'_1 := t,\ y'_2 := y_2 + t^2,\ \ldots,\ y'_m := y_m
+ t^m,\ \lambda' := \sum_{i = 2}^m (-\beta_i -
\alpha_i\frac{G(x)}{F(x)})y_i,
\end{equation}
where $x := (x_0,\ldots,x_3)$. Direct substitution into
\eqref{cubic-sur-x-djd} shows that in fact $\varphi(X_0 \times
\af^{m - 1}) \subseteq \mathcal{X}_0$. It is also clear that
$\varphi$ is \emph{$1$\,-\,to\,-\,$1$ onto its image}.

\begin{lemma}
\label{theorem:l-1} The equality $[X_0]\cdot\mathbb{L}^m =
[\mathcal{X}_0]$ holds in $K_0(\mathrm{Var})$.
\end{lemma}

\begin{proof}
One can see easily with \eqref{mor-2} that $\varphi(X_0 \times
\af^{m - 1})\subseteq \mathcal{X}_0$ is a \emph{closed subscheme}.
Now, since $\varphi$ is injective, we obtain that varieties $X_0
\times \af^{m - 1}$ and $\varphi(X_0 \times \af^{m -
1})_{\text{red}}$ are birational, having the \emph{same}
underlying topological space. Thus we get
$$
[X_0 \times \af^{m - 1}] = [\varphi(X_0 \times \af^{m -
1})_{\text{red}}] = [\varphi(X_0 \times \af^{m - 1})]
$$
according to the relations in $K_0(\mathrm{Var})$ (see
{\ref{subsection:in-1}}).

Next we ``\,move\,'' the scheme $\varphi(X_0 \times \af^{m - 1})$
over $\mathcal{X}_0$ via $\af^{1} \simeq
\text{Spec}\,\fie[\lambda']$:
$$
[x_0:\ldots:x_3]\times (y'_1,y'_2,\ldots,y'_m,\lambda') \mapsto
[x_0:\ldots:x_3]\times (y'_1 +
\lambda',y'_2,\ldots,y'_m,\lambda'(1 - \beta_1 -
\alpha_1\frac{G(x)}{F(x)})) \in \mathcal{X}_0.
$$
By scaling $t$ one may assume that $\beta_1 = 1$ (cf. the
beginning of {\ref{subsection:f-1-aks}}). Thus we obtain a
$1$\,-\,to\,-\,$1$ morphism $\varphi(X_0 \times \af^{m - 1})
\times \af^1 \longrightarrow \mathcal{X}_0$. This yields
$[\varphi(X_0 \times \af^{m - 1})] \cdot [\af^{1}] =
[\mathcal{X}_0]$ and finally
$$
[X_0]\cdot\mathbb{L}^m = [X_0 \times \af^{m}] = [X_0 \times \af^{m
- 1}] \cdot [\af^{1}] = [\varphi(X_0 \times \af^{m - 1})] \cdot
[\af^1] = [\mathcal{X}_0].
$$
\end{proof}

\begin{remark}
\label{remark:-x-k-djdj} Consider the subsets $\mathcal{X}_0^k :=
(y_{m - k + 1} = \ldots = y_m = 0) \cap \mathcal{X}_0 \subset
\mathcal{X}_0$ for $0 < k < m$. Then, arguing as in the proof of
Lemma~\ref{theorem:l-1} with $y_i = -t^{i}$, $m - k + 1 \le i \le
m$, we get $[X_0]\cdot\mathbb{L}^{m - k} = [\mathcal{X}_0^k]$.
\end{remark}

Set $L_i := t \times \af^{m}\cap(y_i = -t^i)$, $1 \le i \le m$,
and $\mathcal{L}_i := \af^{m} \cap (y_i = 0) \simeq L_i$. Note
that $\displaystyle\bigcup_{i = 1}^m L_i$ is the \emph{union of
coordinate hyperplanes} in $\af^m$.

\begin{lemma}
\label{theorem:l-1-djdjd} $[X_0 \times \displaystyle\bigcup_{i =
1}^m L_i] = [\mathcal{X}_0 \cap \displaystyle\bigcup_{i = 1}^m
\mathcal{L}_i]$.
\end{lemma}

\begin{proof}
For any given $L_i$, $2 \le i \le m$, morphism $\varphi$ is
$1$\,-\,to\,-\,$1$ on $X_0 \times L_i \cap (y_1 = t)$ (cf.
\eqref{mor-1} and \eqref{mor-2}). Then, ``\,moving\,'' as in the
proof of Lemma~\ref{theorem:l-1}, yields a $1$\,-\,to\,-\,$1$
morphism $\varphi(X_0 \times L_i \cap (y_1 = t)) \times \af^1
\longrightarrow \mathcal{X}_0 \cap \mathcal{L}_i$. The latter
extends to a $1$\,-\,to\,-\,$1$ morphism $\varphi(X_0 \times
\displaystyle\bigcup_{i = 2}^m L_i \cap (y_1 = t)) \times \af^1
\longrightarrow \mathcal{X}_0 \cap \displaystyle\bigcup_{i = 2}^m
\mathcal{L}_i$.

Further, given $L_i$, $1 \le i \le m - 1$, we similarly define a
morphism $\psi : X_0 \times \af^{m - 1} \longrightarrow
\mathcal{X}_0$,
\begin{equation}
\nonumber [x_0:\ldots:x_3] \times (y_1,\ldots,y_{m-1},t^m) \mapsto
[x_0:\ldots:x_3]\times (y'_1,\ldots,y'_{m-1}, y'_m,\lambda'),
\end{equation}
with
\begin{equation}
\nonumber y'_1 := y_1 + t,\ \ldots,\ y'_{m-1} := y_{m-1} +
t^{m-1},\ y'_m = t^m,\ \lambda' := \sum_{i = 1}^{m-1} (-\beta_i -
\alpha_i\frac{G(x)}{F(x)})y_i,
\end{equation}
which gives a $1$\,-\,to\,-\,$1$ morphism $\psi(X_0 \times
\displaystyle\bigcup_{i = 1}^{m-1} L_i \cap (y_m = t^m)) \times
\af^1 \longrightarrow \mathcal{X}_0 \cap \displaystyle\bigcup_{i =
1}^{m-1} \mathcal{L}_i$.

Finally, we glue the schemes $A := \displaystyle\bigcup_{i = 2}^m
L_i \cap (y_1 = t)$ and $B := \displaystyle\bigcup_{i = 1}^{m - 1}
L_i \cap (y_m = t^m)$ along $C := \displaystyle\bigcup_{i = 2}^{m
- 1} L_i \cap (y_1 = t) \simeq \displaystyle\bigcup_{i = 2}^{m -
1} L_i \cap (y_m = t^m)$, which gives a scheme $A \times_C B$ and
a $1$\,-\,to\,-\,$1$ morphism $\Phi(X_0 \times A \times_C B)
\times \af^1 \longrightarrow \mathcal{X}_0 \cap
\displaystyle\bigcup_{i = 1}^m \mathcal{L}_i$, $\Phi = (\varphi,
\psi)$. Note also that $(A \times_C B \times \af^1)_{\text{red}} =
\displaystyle\bigcup_{i = 1}^m L_i$ and hence we obtain
$$
[\mathcal{X}_0 \cap \displaystyle\bigcup_{i = 1}^m \mathcal{L}_i]
= [\Phi(X_0 \times A \times_C B)]\cdot[\af^1] = [X_0 \times A
\times_C B \times \af^1] = [X_0] \cdot [A \times_C B \times \af^1]
= [X_0 \times \displaystyle\bigcup_{i = 1}^m L_i].
$$
\end{proof}

\refstepcounter{equation}
\subsection{}
\label{subsection:f-2}

Let us consider another locus
$\widetilde{X}\subset\p^3\times\p^1$, given similarly as $X$ with
the same $G,F$, but having other (still generic) degree $m$ forms
$\widetilde{\alpha}(t_0,t_1),\widetilde{\beta}(t_0,t_1)$ (cf.
{\ref{subsection:f-1}}). All the previous gadgets such as
$\mathcal{X}, X_0$, etc. are defined verbatim for $\widetilde{X}$,
and we will distinguish them by simply putting extra
$\,\widetilde{}\,$. We may (and will) also assume that $X$ and
$\widetilde{X}$ are \emph{not isomorphic} to each other --- just
take the forms $\alpha$ and $\widetilde{\alpha}$ from different
$\textbf{PGL}_2(\com)$\,-\,orbits.

One may assume without loss of generality that the fiber of $X
\longrightarrow \p^1$ (resp. of
$\widetilde{X}\longrightarrow\p^1$) over the point $[0:1]$ is a
\emph{smooth} cubic. Note also that all singular fibers in the
pencils $X$ and $\widetilde{X}$ have the same type of
singularities as the surface $(F = 0)$. Indeed, $X$ for instance
corresponds to a smooth rational curve $C$ in the space of all
cubics in $\p^3$, and $F$ is identified with a generic point in
the (discriminant) locus $\Sigma$ of all such \emph{singular}
cubics. It remains to observe that $C$ can be chosen to intersect
$\Sigma$ at generic points (corresponding to cubics with just one
node).

\begin{lemma}
\label{theorem:sing-fib} Both $X$ and $\widetilde{X}$ have the
\emph{same} number of singular fibers.
\end{lemma}

\begin{proof}
Identify $X$ (resp. $\widetilde{X}$) with the curve $C$ (resp.
$\widetilde{C}$) in the space of cubics as in the above argument.
Recall also that $\Sigma$ is a \emph{hypersurface}. Then the
number of singular fibers in $X$ (resp. in $\widetilde{X}$) equals
the intersection index $\Sigma \cdot C$ (resp. $\Sigma \cdot
\widetilde{C}$). We now have $\Sigma \cdot C = \Sigma \cdot
\widetilde{C}$ for $C$ and $\widetilde{C}$ being (obviously)
homologous.
\end{proof}

\begin{remark}
\label{remark:com-for-s} Alternatively, in the proof of
Lemma~\ref{theorem:sing-fib} one could  use the fact that
$\chi_{\text{top}}(X) = \chi_{\text{top}}(\widetilde{X})$,
$\chi_{\text{top}}$ of the cubic surface $S_G := (G = 0) \subset
\mathbb{P}^3$ equals $9$ and $\chi_{\text{top}} = 8$ for $S_F :=
(F = 0)$. Standard topological argument then gives that the number
of singular fibers in $X$ (resp. in $\widetilde{X}$) equals $18 -
\chi_{\text{top}}(X)$ ($= 18 -
\chi_{\text{top}}(\widetilde{X})$).\footnote{~It is also easy to
compute that this equals $32m$.}
\end{remark}

Recall that $S_G$ is the blowup of $\p^2$ at six points. Hence
$[S_G] = [\p^2] + 6\mathbb{L}$ in $K_0(\text{Var})$. Similarly,
$S_F$ is the blowup of $\p^2$ at six points, followed by the
contraction of the proper transform of a line. This gives $[S_F] =
\mathbb{L}^2 + 6\mathbb{L} + 1$. In particular, both $X' :=
X\setminus\{\text{all singular fibers}\}$ and $\widetilde{X}' :=
\widetilde{X}\setminus\{\text{all singular fibers}\}$ satisfy
$[X'] - [\widetilde{X}'] = [X] - [\widetilde{X}]$, with all fibers
of $X'$ (resp. of $\widetilde{X}'$) having the same class in
$K_0(\text{Var})$. It is tempting to propose at this point that
$[X'] = [\widetilde{X}']$. However, there is no rigorous proof of
this, and instead we argue via Proposition~\ref{theorem:p-1}
below.

\refstepcounter{equation}
\subsection{}
\label{subsection:f-2-djdjdj}

Let us consider the schemes $S := (GF = 0) \cap X$ and $Z := (G =
F = 0) \subset \mathbb{P}^3 \times \mathbb{P}^1$. Note that $S$
equals the union of $Z$ and the \emph{disjoint} union of $m$
copies of surfaces $S_G$ and $S_F$ (cf.
Remark~\ref{remark:com-for-s}). Hence $S$ is independent of $X$
and $\widetilde{X}$.

After all the setup we prove the following:

\begin{prop}
\label{theorem:p-1} $[X]\cdot(\mathbb{L} - 1)^m =
[\widetilde{X}]\cdot(\mathbb{L} - 1)^m$ and $[X]\cdot\mathbb{L}^k
= [\widetilde{X}]\cdot\mathbb{L}^k$ for all $1\le k\le m$.
\end{prop}

\begin{proof}
Firstly, since $$[X] = [X_0] + [\text{the fiber over}\ [0:1]] +
[S] = [X_0] + [S_G] + [S]$$ (same for $\widetilde{X}$), it
suffices to consider $X_0$ and $\widetilde{X}_0$ in place of $X$
and $\widetilde{X}$, respectively.

Further, $\mathcal{X}_0 \simeq \widetilde{\mathcal{X}}_0$ because
one may perform a linear transformation of
$y_1,\ldots,y_m,\lambda$ that brings $L_{\alpha}$ to
$L_{\widetilde{\alpha}}$ and $L_{\beta} + \lambda$ to
$L_{\widetilde{\beta}} + \lambda$. Then Lemma~\ref{theorem:l-1}
gives $[X_0]\cdot \mathbb{L}^{k} = [\widetilde{X}_0]\cdot
\mathbb{L}^{k}$ for $k = m$. The cases of other $0 < k < m$ are
similar
--- take $\mathcal{X}_0^k$ $(\simeq \widetilde{\mathcal{X}}_0^k)$
as in Remark~\ref{remark:-x-k-djdj}.

Finally, we have
$$
[X_0] \cdot(\mathbb{L} - 1)^m = [X_0 \times \af^m \setminus
\displaystyle\bigcup_{i = 1}^m L_i] = [X_0 \times \af^m] - [X_0
\times \displaystyle\bigcup_{i = 1}^m L_i]
$$
in the notation of {\ref{subsection:f-1-aks}}, with $[X_0 \times
\displaystyle\bigcup_{i = 1}^m L_i] = [\mathcal{X}_0 \cap
\displaystyle\bigcup_{i = 1}^m \mathcal{L}_i]$ according to
Lemma~\ref{theorem:l-1-djdjd} (same for $\widetilde{X}_0$ and
$\widetilde{\mathcal{X}}_0$ of course). Note also that
$[\mathcal{X}_0 \cap \displaystyle\bigcup_{i = 1}^m \mathcal{L}_i]
= [\widetilde{\mathcal{X}}_0 \cap \displaystyle\bigcup_{i = 1}^m
\mathcal{L}_i]$ by the inclusion\,-\,exclusion principle because
$\mathcal{X}_0 \cap \mathcal{L}_{i_1} \cap \ldots \cap
\mathcal{L}_{i_k} \simeq \widetilde{\mathcal{X}}_0 \cap
\mathcal{L}_{i_1} \cap \ldots \cap \mathcal{L}_{i_k}$ for all $1
\le i_1 \le \ldots \le i_k \le m$ and $1 \le k \le m$. Thus we
obtain
$$
[X_0] \cdot(\mathbb{L} - 1)^m = [X_0] \cdot \mathbb{L}^m -
[\mathcal{X}_0 \cap \displaystyle\bigcup_{i = 1}^m \mathcal{L}_i]
= [\widetilde{X}_0]\cdot \mathbb{L}^{m}
-[\widetilde{\mathcal{X}}_0 \cap \displaystyle\bigcup_{i = 1}^m
\mathcal{L}_i] = [\widetilde{X}_0 \times \af^m] - [\widetilde{X}_0
\times \displaystyle\bigcup_{i = 1}^m L_i] =
[\widetilde{\mathcal{X}}_0] \cdot(\mathbb{L} - 1)^m
$$
and Proposition~\ref{theorem:p-1} is completely proved.
\end{proof}

Expand $(\mathbb{L} - 1)^m = \mathbb{L}^m - m\mathbb{L}^{m-1} +
\ldots + (-1)^m$ in (commutative) $K_0(\text{Var})$ and multiply
by $[X]$ (resp. by $[\widetilde{X}]$). Then it follows from
Proposition~\ref{theorem:p-1} that $[X] = [\widetilde{X}]$. Hence
the first condition in Theorem~\ref{theorem:main} is satisfied for
the two $3$\,-\,folds.

\begin{lemma}
\label{theorem:non-bir} The $3$\,-\,folds $X$ and $\widetilde{X}$
are \emph{not} birational to each other.\footnote{~Note at the
same time that $X$ and $\widetilde{X}$ are \emph{stably
birational} according to \cite[Corollary 2.6]{la-lu}.}
\end{lemma}

\begin{proof}
The pencils $X$ and $\widetilde{X}$ are particular cases of del
Pezzo fibrations studied in \cite{pux} (cf. our
Lemma~\ref{theorem:x-is-smooth}). Furthermore, $X$ and
$\widetilde{X}$ satisfy two genericity assumptions from \cite[\S
1]{pux}, concerning the pencils of degree $3$ del Pezzo surfaces.
Namely, the assumption about singular fibers (cf. \textsc{Remark}
in \emph{loc. cit}) has been met in {\ref{subsection:f-2}} above,
whereas the \emph{$K^2$\,-\,condition} translates precisely to $m
\ge 3$ in {\ref{subsection:f-1}}. Thus it remains to apply
\cite[Corollary 2.1,\,(i)]{pux}, recalling that $X$,
$\widetilde{X}$ were chosen non\,-\,isomorphic.
\end{proof}

Lemma~\ref{theorem:non-bir} completes the proof of
Theorem~\ref{theorem:main}.

\bigskip

\thanks{{\bf Acknowledgments.} I am grateful to A. Bondal, S. Galkin, M. Kapranov, V. Lunts, B. Poonen, A. Pukhlikov, and J. Sebag
for their interest, helpful discussions and corrections. I also
thank the anonymous referee whose comments and suggestions have
considerably improved the exposition. The work was supported by
World Premier International Research Initiative (WPI), MEXT,
Japan, and Grant\,-\,in\,-\,Aid for Scientific Research (26887009)
from Japan Mathematical Society (Kakenhi), and by the Russian
Academic Excellence Project 5\,-\,100.

\end{document}